\documentclass[letterpaper,11pt]{amsart}

\usepackage[all]{xy}                        %

\CompileMatrices                            % Faster

\UseTips                                    % Use

\input xypic
\usepackage[bookmarks=true]{hyperref}       % Hyperref
\usepackage{amssymb,latexsym,amsmath,amscd}
\usepackage{xspace}

\usepackage{graphicx}
\reversemarginpar

\theoremstyle{plain}
\newtheorem{theorem}{Theorem}[section]
\newtheorem*{theorem*}{Theorem}
\newtheorem{proposition}[theorem]{Proposition}
\newtheorem{corollary}[theorem]{Corollary}
\newtheorem{lemma}[theorem]{Lemma}

\theoremstyle{definition}
\newtheorem{definition}[theorem]{Definition}

\newtheorem{remark}[theorem]{Remark}

\newcommand{\enm}[1]{\ensuremath{#1}}          %
\newcommand{\op}[1]{\operatorname{#1}}
\newcommand{\cal}[1]{\mathcal{#1}}

\newcommand{\CC}{\enm{\mathbb{C}}}

\newcommand{\ZZ}{\enm{\mathbb{Z}}}

\newcommand{\PP}{\enm{\mathbb{P}}}

\newcommand{\Aa}{\enm{\cal{A}}}           % All
\newcommand{\Dd}{\enm{\cal{D}}}

\newcommand{\Hh}{\enm{\cal{H}}}
\newcommand{\Ii}{\enm{\cal{I}}}

\newcommand{\Mm}{\enm{\cal{M}}}

\newcommand{\Oo}{\enm{\cal{O}}}

\newcommand{\Uu}{\enm{\cal{U}}}

\renewcommand{\phi}{\varphi}
\renewcommand{\theta}{\vartheta}
\renewcommand{\epsilon}{\varepsilon}

\newcommand{\Hom}{\op{Hom}}
\newcommand{\Ext}{\op{Ext}}
\newcommand{\Sym}{\op{Sym}}

\newcommand{\End}{\op{End}}

\renewcommand{\a}{\alpha}

\newcommand{\old}[1]{}

\begin{document}

%\layout
\title[Jumping Conics on a Smooth Quadric in $\PP_3$]{Jumping Conics \\on a Smooth Quadric in $\PP_3$}

\author{Sukmoon Huh}
\address{Korea Institute for Advanced Study \\
Hoegiro 87, Dongdaemun-gu \\
Seoul 130-722, Korea}
\email{sukmoonh@kias.re.kr}
%\thanks{??}

\maketitle
%\tableofcontents

\begin{abstract}
We investigate the jumping conics of stable vector bundles $E$ of rank 2 on a smooth quadric surface $Q$ with the first Chern class $c_1=\Oo_Q(-1,-1)$ with respect to the ample line bundle $\Oo_Q(1,1)$. We show that the set of jumping conics of $E$ is a hypersurface of degree $c_2(E)-1$ in $\PP_3^*$. Using these hypersurfaces, we describe moduli spaces of stable vector bundles of rank 2 on $Q$ in the cases of lower $c_2(E)$.
\end{abstract}

\section{Introduction}
The moduli space of stable sheaves on surfaces has been studied by many people. Especially, over the projective plane, the moduli space of stable sheaves of rank 2 was studied by W.Barth \cite{Barth2} and K.Hulek \cite{Hulek}, using the jumping lines and jumping lines of the second kind. In \cite{Vitter}, this idea was generalized to the jumping conics on the projective plane. In this article, we use the concept of jumping conics on the smooth quadric surface.

Let $Q$ be a smooth quadric in $\PP_3=\PP (V)$, where $V$ is a 4-dimensional vector space over complex numbers $\CC$, and $\Mm(k)$ be the moduli space of stable vector bundles of rank 2 on $Q$ with the Chern classes $c_1=\Oo_Q(-1,-1)$ and $c_2=k$ with respect to the ample line bundle $H=\Oo_Q(1,1)$. $\Mm (k)$ form an open Zariski subset of the projective variety $\overline{\Mm}(k)$ whose points correspond to the semi-stable sheaves on $Q$ with the same numerical invariants. The Zariski tangent space of $\Mm(k)$ at $E$, is naturally isomorphic to $H^1(Q, \End (E))$ and so the dimension of $\Mm (k)$ is equal to $h^1(Q, \End (E))=4k-5$, since $E$ is simple.

Using the Beilinson-type theorem on $Q$ \cite{Buch}, we obtain the following monad for $E\in \Mm(k)$,
$$0\rightarrow \CC^{k-1}\otimes \Oo_Q(-1,-1) \rightarrow \CC^{k} \otimes (\Oo_Q(0,-1)\oplus \Oo_Q(-1,0)) \rightarrow \CC^{k-1}\otimes \Oo_Q \rightarrow 0,$$
with the cohomology sheaf $E$, where the first injective map derives a map
$$\delta : H^1(E(-1,-1))\otimes V^* \rightarrow H^1(E).$$
As in \cite{Barth}, we similarly define $S(E)\subset \PP_3^*$, the set of jumping conics of $E$, and prove that $S(E)$ is a hypersurface in $\PP_3^*$ of degree $k-1$ whose equation is given by $\det \delta(z)=0$, $z\in V^*$, where $\delta(z)$ is a symmetric $(k-1)\times (k-1)$-matrix. We give a criterion for $H\in \PP_3^*$ to be a singular point of $S(E)$ and calculate the exact number of singular points of $S(E)$ when $E$ is a Hulsbergen bundle, i.e. $E$ admits the following exact sequence,
$$0\rightarrow \Oo_Q \rightarrow E(1,1) \rightarrow I_Z(1,1) \rightarrow 0,$$
where $Z$ is a 0-cycle on $Q$ with length $k$ whose support is in general position.

In Section 4, we describe the above results in the cases $c_2\leq 3$ by investigating the map
$$S : \Mm (k) \rightarrow |\Oo_{\PP_3^*}(k-1)|,$$
sending $E$ to $S(E)$. When $c_2=2$, $S(E)$ is a hypersurface in $\PP_3$ and $\Mm(2)$ is isomorphic to $\PP_3\backslash Q$ via $S$, which was already shown in \cite{Huh}. In the case of $c_2=3$, we investigate the surjective map from $\Mm(3)$ to $\PP_3^*$, sending $E$ to the vertex point of the quadric cone $S(E)\subset \PP_3^*$ to give an explicit description of $\Mm (3)$. In fact, the generic fibre of this map over $H\in \PP_3^*$ is isomorphic to the set of smooth conics which are Poncelet related to the smooth conic $H\cap Q$. As a result, we can observe that $S$ is an isomorphism from $\Mm(3)$ to its image and in particular, when $c_2=2,3$, the set of jumping conics, $S(E)$, determines $E$ uniquely.

We would like to thank E. Ballico, F. Catanese, I. Dolgachev and M. Reid for many suggestions and advices.

\section{The Beilinson Theorem and Jumping Conics}
\subsection{The Beilinson Theorem}
Let $V_1$ and $V_2$ be two 2-dimensional vector spaces with the coordinate $[x_{1i}]$ and $[x_{2j}]$, respectively. Let $Q$ be a smooth quadric isomorphic to $\PP(V_1)\times \PP(V_2)$ and then it is embedded into $\PP_3\simeq \PP(V)$ by the Segre map, where $V=V_1\otimes V_2$. Let us denote $f^*\Oo_{\PP_1}(a) \otimes g^*\Oo_{\PP_1}(b)$ by $\Oo_Q(a,b)$ and $E\otimes \Oo_Q(a,b)$ by $E(a,b)$ for coherent sheaves $E$ on $Q$, where $f$ and $g$ are the projections from $Q$ to each factors. Then the canonical line bundle $K_Q$ of $Q$ is $\Oo_Q(-2,-2)$.
\begin{definition}
For a fixed ample line bundle $H$ on $Q$, a torsion free sheaf $E$ of rank $r$ on $Q$ is called \textit{stable} (resp. \textit{semi-stable}) with respect to $H$ if
$$\frac{\chi(F \otimes \Oo_Q(mH))}{r'} < (resp. \leq ) \frac{\chi(E\otimes \Oo_Q(mH))}r,$$
for all non-zero subsheaves $F\subset E$ of rank $r'$.
\end{definition}
Let $\overline{\Mm} (k)$ be the moduli space of semi-stable sheaves of rank 2 on $Q$ with the Chern classes $c_1=\Oo_Q(-1,-1)$ and $c_2=k$ with respect to the ample line bundle $H=\Oo_Q(1,1)$. The existence and the projectivity of $\overline{\Mm}(k)$ is known in \cite{Gieseker} and it has an open Zariski subset $\Mm(k)$ which consists of the stable vector bundles with the given numeric invariants. By the Bogomolov theorem, $\Mm (k)$ is empty if $4k < c_1^2=2$ and in particular, we can consider only the case $k\geq 1$. Note that $E \simeq E^*(-1,-1)$ and by the Riemann-Roch theorem, we have
$$\chi_E(m):=\chi(E(m,m))=2m^2+2m+1-k,$$
for $E\in \overline{\Mm}(k)$.

Using the same trick as in the proof of the Beilinson theorem on the vector bundles over the projective space \cite{OSS}, we can obtain similar statement over $Q$.

\begin{proposition}\cite{Buch}
For any holomorphic bundle $E$ on $Q$, there is a spectral
sequence
$$E_1^{p,q} \Rightarrow E_{\infty}^{p+q}=\left\{
                                           \begin{array}{ll}
                                             E, & \hbox{if $p+q=0$;} \\
                                             0, & \hbox{otherwise,}
                                           \end{array}
                                         \right.$$
with
$$\left\{
    \begin{array}{ll}
      E_1^{p,q}=0, ~~|p+1|>1& \hbox{} \\
      E_1^{0,q}=H^q(E)\otimes \Oo_Q & \hbox{} \\
      E_1^{-2,q}=H^q(E(-1,-1))\otimes \Oo_Q (-1,-1), & \hbox{}
    \end{array}
  \right.
$$
and an exact sequence
$$\cdots \rightarrow H^q(E(0,-1))\otimes \Oo_Q (0,-1) \rightarrow
E_1^{-1,q} \rightarrow H^q(E(-1,0))\otimes \Oo_Q (-1,0) \rightarrow
\cdots.$$
\end{proposition}
\begin{proof}
Let $p_1$ and $p_2$ be the projections from $Q\times Q$ to each factors and denote $p_1^*\Oo_Q(a,b)\otimes p_2^*\Oo_Q(c,d)$ by $\Oo (a,b)(c,d)'$. If we let $\triangle$ be the diagonal of $Q\times Q$, we have the following Koszul complex,
\begin{equation}
0\rightarrow \Oo(-1,-1)(-1,-1)' \rightarrow \bigoplus_{i=0}^1 \Oo (-i,1-i)(-i,1-i)' \rightarrow \Oo \rightarrow \Oo_{\triangle}.
\end{equation}
If we tensor it with $p_2^*E$, then we have a locally free resolution of $p_2^*E|_{\triangle}$. If we take higher direct images under $p_1$, we get the assertion by the standard argument on the spectral sequence.
\end{proof}

From the stability condition of $E\in \Mm(k)$, we have $H^0(E(a,b))=0$ whenever $a+b\leq 0$. Hence $E_1^{p,q}=0$ for
$p=-2,-1,0$ and $q=0,2$ and thus the proposition gives us a monad
\begin{equation}
M: 0\rightarrow K_{1,1}\otimes \Oo_Q (-1,-1) \rightarrow E_1^{-1,1}
\rightarrow K_{0,0} \otimes \Oo_Q \rightarrow 0,
\end{equation}
with the cohomology sheaf $E(M)=E$, where $K_{a,b}=H^1(E(-a,-b))$ and $E_1^{-1,1}$ fits into
the following exact sequence,
\begin{equation}
0\rightarrow K_{0,1} \otimes \Oo_Q (0,-1) \rightarrow
E_1^{-1,1} \rightarrow K_{1,0} \otimes \Oo_Q(-1,0) \rightarrow 0.
\end{equation}
Since $H^1(\Oo_Q (1,-1))=0$, this exact sequence splits. Thus we have the following corollary.

\begin{corollary}\label{monad}
Let $E\in \Mm(k)$. Then $E$ becomes the cohomology sheaf of the following monad:
\begin{align*}
M(E)~~:~~0\rightarrow &K_{1,1} \otimes \Oo_Q(-1,-1) \rightarrow \\  &\bigoplus_{i=0}^1 (K_{i,1-i}\otimes \Oo_Q (-i,-1+i)) \rightarrow K_{0,0}\otimes \Oo_Q \rightarrow 0.
\end{align*}
\end{corollary}

Note that $k_{1,1}=k_{0,0}=k-1$ and $k_{1,0}=k_{0,1}=k$, where $k_{i,j}=\dim K_{a,b}$.

Let us denote by $a$, the first injective map in the monad in the corollary (\ref{monad}). Since $E\simeq E^*(-1,-1)$, the last surjective map is the dual of $a$, twisted by $\Oo_Q(-1,-1)$ and thus the monad $M$ is completely determined by $a$. The monomorphism $a$ corresponds to an element $\a$ in $$K_{1,1}^* \otimes ((K_{0,1}\otimes V_1)\oplus (K_{1,0}\otimes
V_2)),$$
i.e. $\a=(\a_1, \a_2)$, where $\a_i\in \Hom (V_i^*, \Hom (K_{1,1}, K_{i-1,2-i}))$. Since $K_{1,1}^*\simeq K_{0,0}$ and $K_{1,0}^*\simeq K_{0,1}$, we can obtain a map
\begin{equation}
\delta : V_1^* \otimes V_2^*  \rightarrow \Hom (K_{1,1} , K_{0,0}),
\end{equation}
defined by $\delta := \a_2^t \circ \a_1 + \a_1^t \circ a_2$. So $\delta(z)\in K_{0,0}\otimes K_{0,0}$ since $K_{1,1}^*\simeq K_{0,0}$. Again from the self-duality of $E$, i.e. $E\simeq E^*(-1,-1)$, we have $\delta(z)=\delta (z) ^t$. In other words, $\delta(z)$ is an element in $\Sym^2(K_{0,0})$ for all $z$.

\subsection{Jumping Conics}
Let $H$ be a general hyperplane section of $\PP_3$ and then $C_H:=Q\cap H$ is a conic on $H$. Let $E$ be a vector bundle of rank $r$ on $Q$. If we choose an isomorphism $f: \PP_1 \rightarrow C_H$, then due to Grothendieck, we have
$$f^* E|_{C_H} \simeq \Oo_{\PP_1} (a_1) \oplus \cdots \oplus \Oo_{\PP_1}(a_r),$$
where $a_{E,H}:=(a_1, \cdots, a_r) \in \ZZ ^r$ such that $a_1\geq \cdots \geq a_r$. Here, $a_{E,H}$ is called the splitting type of $E|_{C_H}$.

\begin{definition}
 A conic $C_H=Q\cap H$ on $Q$, is called \textit{a jumping conic} of $E$ if the splitting type $a_{E,H}$ of $E|_{C_H}$ is different from the generic splitting type $a_E$. We will denote the set of jumping conics of $E$ by $S(E)\subset \PP_3^*$.
\end{definition}
\begin{remark}
The above definition is valid only for the general hyperplane sections $H$. Later, we give an equivalent definition for the jumping conics for arbitrary case, using the cohomological criterion.
\end{remark}
From the theorem (0.2) in \cite{Maruyama}, we have $a_i-a_{i+1}\leq 2$ for all $i$ since the degree of $Q\subset \PP_3$ is 2. From the following proposition, we know that this upper bound can be sharpened to be 1.

\begin{proposition}\label{GM}
If $E$ is a stable vector bundle on $Q$ of rank $r$, we have
$$a_i-a_{i+1}\leq 1, \text{    for all } i,$$
where $a_E=(a_1, \cdots, a_r)$.
\end{proposition}
\begin{proof}
The proof is analogue of the one for the Grauert-M\"{u}lich theorem in \cite{OSS}. We consider the incidence variety $\mathbf{I}=\{(x,H)\in Q \times \PP_3^*~|~ x\in C_H\}$ with the projections $\pi_1$ and $\pi_2$ to each factors. Suppose that $i$ is the first index such that $a_i-a_{i+1}\geq 2$. Moreover, we can assume that $a_i=0$. Now, let us consider the natural map
\begin{equation}
\pi_2^* {\pi_2}_* \pi_1^* E \rightarrow \pi_1^* E
\end{equation}
and $E_1$ be the image of this map. Then $E_1$ is a subsheaf of $\pi_1^* (E)$ on $\mathbf{I}$ of rank $i$ such that
$$f^*E_1|_{\pi_2^{-1}(H)}\simeq \Oo_{\PP_1}(a_1)\oplus \cdots \Oo_{\PP_1}(a_i),$$
for a general $H\in \PP_3^*$ and an isomorphism $f: \PP_1 \rightarrow C_H$. Then the quotient sheaf $E_2:=\pi_1^*(E)/E_1$ is of rank $r-i$ with
$$f^*E_2|_{\pi_2^{-1}(H)}\simeq \Oo_{\PP_1}(a_{i+1})\oplus \cdots \Oo_{\PP_1}(a_r),$$
for a general $H\in \PP_3^*$. From the following lemma, the pull-back of the relative tangent bundle $T_{\mathbf{I}|Q}$ to $\PP_1$, is $\Oo_{\PP_1}(-1)^{\oplus 2}$. Hence, the restriction of the sheaf $\Hh om(T_{\mathbf{I}|Q}, \Hh om(E_1, E_2))$ to $C_H$ is isomorphic to the direct sum of $\Oo_{\PP_1}(a_{j_1}-a_{j_2}+1)^{\oplus 2}$ where $j_1\geq i+1$ and $j_2\leq i$. In particular, we have
$$\Hom (T_{\mathbf{I}|Q}, \Hh om(E_1, E_2))=0.$$
By the Descente-Lemma \cite{OSS}, there exists a subsheaf of $E'$ of $E$ on $Q$ such that $\pi_1^*E'=E_1$ and it would make the contradiction to the stability of $E$.
\end{proof}
\begin{lemma}
Let $f: \PP_1 \rightarrow C_H$ be an isomorphism. Then we have am isomorphism
$$f^* T_{\mathbf{I}|Q}|_{C_H}\simeq \Oo_{\PP_1}(-1)^{\oplus 2}.$$
\end{lemma}
\begin{proof}
Let $\mathbf{I}'$ be the incidence variety in $\PP_3 \times \PP_3^*$, i.e. $\mathbf{I}'\simeq \PP(T^*_{\PP_3})$. Then we have the universal exact sequence of $\PP(T^*_{\PP_3})$,
$$0\rightarrow M \rightarrow \pi_1^* T^*_{\PP_3} \rightarrow N \rightarrow 0,$$
where $M$ and $N$ are vector bundles on $\mathbf{I}'$ of rank 1 and 2, respectively. If we restrict the universal sequence to $\pi_2^{-1}(\{H\})$, then we obtain
$$0\rightarrow N_{H|\PP_3}^* \rightarrow T_{\PP_3}^*|_H \rightarrow T_H^* \rightarrow 0,$$
where $N_{H|\PP_3}\simeq \Oo_H(1)$ is the normal bundle of $H$ in $\PP_3$. Note that
$$T_{\mathbf{I}'|\PP_3}|_H \simeq \Hh om(M, N)|_H\simeq N_{H|\PP_3}\otimes T_H^*.$$
Since $f^*N_{H|\PP_3}$ is isomorphic to $\Oo_{\PP_1}(2)$, it is enough to prove that
$$f^*T_H^* \simeq f^* \Omega_H \simeq\Oo_{\PP_1}(-3)^{\oplus 2}$$
since $T_{\mathbf{I}|\PP_3}$ is the restriction of $T_{\mathbf{I}'|\PP_3}$ to $\mathbf{I}$.
If we tensor the following exact sequence,
$$0\rightarrow \Oo_{\PP_2}(-2) \rightarrow \Oo_{\PP_2}\rightarrow \Oo_{C_H} \rightarrow 0,$$
with $\Omega_{\PP_2}(1)$ (recall that $H\simeq \PP_2$ and $C_H$ is the image conic of $\PP_1$ by $f$) and take the long exact sequence of cohomology, then we obtain $H^0(\Omega_{\PP_2}(1)|_{C_H})=0$ from the Bott theorem \cite{OSS}. Thus we have
$$h^0(f^*(\Omega_{\PP_2}(1)))=h^0(\Omega_{\PP_2}(1)|_{C_H})=0.$$
Since $c_1(f^*(\Omega_{\PP_2}(1)))=-2$, the only possibility is that $f^*(\Omega_{\PP_2}(1))\simeq \Oo_{\PP_1}(-1)^{\oplus 2}$ and so $f^*T_H^*\simeq \Oo_{\PP_1}(-3)^{\oplus 2}$.
\end{proof}
\begin{remark}
From the theorem 1 in \cite{Vitter}, we know that for a semistable vector bundle of rank 2 on $\PP_2$ and a general smooth conic $f: \PP_1 \hookrightarrow \PP_2$, we have
$$f^* E \simeq \left\{
                 \begin{array}{ll}
                   \Oo_{\PP_1}^{\oplus 2}, & \hbox{if $c_1(E)=0$;} \\
                   \Oo_{\PP_1}(-1)^{\oplus 2}, & \hbox{if $c_1(E)=-1$.}
                 \end{array}
               \right.$$
Since $T_{\PP_2}$ is semistable, we can also obtain $f^*T_{\PP_2}^* \simeq \Oo_{\PP_1}(-3)^{\oplus 2}$ for a general conic. In fact, this is true for all conics and $T_{\PP_2}$ is the only vector bundle of rank 2 on $\PP_2$ up to twists with the same splitting type over all smooth conics except a direct sum of two line bundles \cite{IM}.
\end{remark}

Let us assume that $E\in \Mm(k)$ be a stable vector bundle on $Q$. As an immediate consequence of the above proposition, we get the following corollary.

\begin{corollary}
For a general conic $C_H$ on $Q$, we have
$$E|_{C_H}\simeq \Oo_{C_H}(-p)\oplus \Oo_{C_H}(-p),$$
where $p$ is a point on $C_H$.
\end{corollary}

In particular, the jumping conics of $E$ can be characterized by
\begin{equation}
h^0(E|_{C_H})\not= 0,
\end{equation}
and we will use this cohomological criterion as the definition of the jumping conics of $E$.

 We consider the exact sequence,
$$0\rightarrow E(-1,-1) \rightarrow E \rightarrow E|_{C_H} \rightarrow 0,$$
to derive the following long exact sequence,
\begin{equation}
0\rightarrow H^0(E|_{C_H}) \rightarrow H^1(E(-1,-1)) \rightarrow H^1(E),
\end{equation}
where the last map is given by $\delta(z)= \a_2^t \otimes \a_1 + \a_1^t \otimes \a_2$, where $z$ is the coordinates determining the hyperplane section $H$. Hence $C_H$ is a jumping conic if and only if $\det (\delta)=0$. Note that $\det(\delta)$ is a homogeneous polynomial of degree $c_2(E)-1$ with the coordinates of $V_1^* \otimes V_2^*$. This determinant does not vanish identically due to the proposition (\ref{GM}), and so we obtain that $S(E)$ is a hypersurface of degree $c_2(E)-1$ in $\PP_3^*$. From the fact that $\delta(z)=\delta (z) ^t$, we obtain the following statement.

\begin{theorem}\label{main}
$S(E)$ is a symmetric determinantal hypersurface of degree $c_2(E)-1$ in $\PP_3^*$.
\end{theorem}

The natural question on $S(E)$ is the smoothness and the next proposition will give an answer to this question.
\begin{proposition}
If $h^0(E|_{C_H})\geq 2$, then $H\in \PP_3^*$ is a singular point of $S(E)$.
\end{proposition}
\begin{proof}
The statement is clear from the theory on the singular locus of symmetric determinantal varieties \cite{Harris}. Indeed, let $M=M_0$ denote the projective space $\PP_N$ of $(k-1)\times (k-1)$ matrices up to scalars and $M_i$ be the locus of matrices of corank $i$ or more. Let us consider a map $\phi : \PP_3^* \rightarrow M$, determined naturally by $\delta$. If we let $S_i$ be the preimage of $M_i$ via $\phi$, then we have
\begin{align*}
T_p S_2 &= d\phi ^{-1} (T_q \phi(S_2))\\
        &= d\phi^{-1} (T_q M_2\cap T_q\phi (\PP_3^*))\\
        &= d \phi^{-1} (M \cap T_q \phi (\PP_3^*))~~~~~\text{,since $T_qM_2=M$ \cite{Harris}}\\
        &= d \phi^{-1}T_q \phi (\PP_3^*) = \PP_3^*,
\end{align*}
where $q=\phi(p)$ and $p\in S_2$. In particular, $S_2$ is the singular locus of $S_1=S(E)$.
\end{proof}
\begin{remark}
Let $f: \PP_1 \rightarrow C_H\subset Q$ be a smooth conic on $Q$ and assume that we have
$$f^*E|_{C_H} \simeq \Oo_{\PP_1}(-1-i) \oplus \Oo_{\PP_1}(-1+i), $$
where $i$ is a nonnegative integer. Note that $k=h^0(E|_{C_H})=\text{corank}(\delta (z))$, where $z$ is the coordinates of $H$. If $i\geq 2$, then $H\in \PP_3^*$ is a singular point of $S(E)$.
\end{remark}

Now for later use, let us define a sheaf supported on $S(E)$. As in \cite{Barth2}, we can see that $S(E)$ is the support of the $\Oo_{\PP_3^*}$-sheaf $\theta_E(1)$ defined by the following exact sequence,
\begin{equation}\label{te}
0\rightarrow K_{1,1}\otimes \Oo_{\PP_3^*}(-1) \rightarrow K_{0,0}\otimes \Oo_{\PP_3^*} \rightarrow \theta_E(1) \rightarrow 0.
\end{equation}
The first injective map is composed of
$$K_{1,1}\otimes \Oo_{\PP_3^*}(-1) \rightarrow K_{1,1}\otimes (V_1^*\otimes V_2^*) \otimes \Oo_{\PP_3^*} \rightarrow K_{0,0} \otimes \Oo_{\PP_3^*},$$
where the first map is from the Euler sequence over $\PP_3^*$ and the second map is from the map $\delta$. So $\theta_E$ is an $\Oo_{S(E)}$-sheaf.

From the incidence variety $\mathbf{I}\subset Q \times \PP_3^*$, we obtain
$$0\rightarrow \pi_1^*\Oo_Q(-1,-1) \otimes \pi_2^* \Oo_{\PP_3^*}(-1) \rightarrow \Oo_{Q\times \PP_3^*} \rightarrow \Oo_{\mathbf{I}} \rightarrow 0.$$
If we tensor it with $\pi_1^*E$ and take the direct image of it, we obtain,
$$0\rightarrow K_{1,1}\otimes \Oo_{\PP_3^*}(-1) \rightarrow K_{0,0}\otimes \Oo_{\PP_3^*} \rightarrow R^1{\pi_2}_*\pi_1^*E \rightarrow 0.$$
Since this exact sequence coincide with the sequence (\ref{te}), we have
\begin{lemma}
$\theta_E(1) \simeq R^1{\pi_2}_*\pi_1^*E.$
\end{lemma}

%\begin{theorem}
%The bundle $E$ on $Q$ is completely determined by the pair $(S(E), \theta_E)$.
%\end{theorem}
%\begin{proof}
%The proof of the theorem 7.3 in \cite{Hulek} works in our case with similar argument.
%\end{proof}

\section{Examples}
Let $\Mm(k)$ be the moduli space of stable vector bundles of rank 2 on $Q$ with the Chern classes $c_1=\Oo_Q(-1,-1)$ and $c_2=k$ with respect to the ample line bundle $\Oo_Q(1,1)$. The dimension of $\Mm(k)$ can be computed to be $4k-5$. By sending $E\in \Mm (k)$ to the set of jumping conics of $E$, we can define a morphism
$$S : \Mm(k) \rightarrow |\Oo_{\PP_3^*}(k-1)|\simeq \PP_N,$$
where $N={{k+2}\choose 3}-1$.

Let $Z=\{x_1,\cdots, x_k\}$ be a 0-dimensional subscheme of $Q$ with length $k$ in general position. If $E$ is a stable vector bundle fitted into the exact sequence,
$$0\rightarrow \Oo_Q \rightarrow E(1,1) \rightarrow I_Z(1,1) \rightarrow 0,$$
which is called \textit{a Hulsbergen bundle}, then $E$ is in $\Mm(k)$. Note that if $k\leq 4$, then $E\in \Mm(k)$ admits the above exact sequence. Conversely, let us consider the above extension. It is parametrized by
$$\PP (Z):=\PP \Ext^1 (I_Z(1,1), \Oo_Q)\simeq \PP H^0(\Oo_Z)^*.$$
If we give $\PP (Z)$ the coordinate system $(c_1, \cdots, c_k)$ corresponding to $Z$, then by the lemma (5.1.2) in Chapter 1 \cite{OSS} or \cite{C}, the bundle $E$ corresponding to $(c_1, \cdots, c_k)$ is locally free if and only if $c_i\not= 0$ for all $i$.

Now by the theorem (\ref{main}), $S(E)\subset \PP_3^*$ is a hypersurface of degree $k-1$.
\begin{lemma}${ } $
\begin{enumerate}
\item If $|Z\cap H|\geq 3$, then $h^0(E|_{C_H})\geq 2$.
\item If $|Z\cap H|\leq 2$, then $h^0(E|_{C_H})\leq 1$.
\end{enumerate}
\end{lemma}
\begin{proof}
Let $m=|Z\cap H|\geq 3$. If $C_H$ is a smooth conic, then by tensoring the above exact sequence with $\Oo_H$, we have $E|_{C_H}\simeq \Oo_{C_H}((m-1)p)\oplus \Oo_{C_H}(-mp)$ since $\Ext^1 (\Oo_{C_H}(-mp), \Oo_{C_H}((m-1)p))=0$. Thus, $h^0(E_{C_H})\geq 2$.

Let us assume that $C_H=l_1+l_2$, i.e. $H$ is a tangent plane of $Q$. Note that
$$h^0(C_H,\Oo(a_1,a_2))= \left\{
                                                     \begin{array}{ll}
                                                       0, & \hbox{if $a_i<0$;} \\
                                                       a_i, & \hbox{if $a_i\geq 0$, $a_j<0$;} \\
                                                       a_1+a_2+1, & \hbox{if $a_i\geq0$}
                                                     \end{array}
                                                   \right.
$$
where $\Oo(a_1,a_2):=\Oo_{l_1}(a_1)\cup \Oo_{l_2}(a_2)$. From the lemma (2.1) in \cite{IM}, it is clear that $h^0(E_{C_H})\geq 2$. For example, when $m=3$ and $Z\cap H=\{x,y,z\}$, $x,y\in l_1$, $z\in l_2$ and $q=l_1\cap l_2\not\in Z$, we have
\begin{equation}\label{eq4}
0\rightarrow \Oo_{l_1}(1)\cup\Oo_{l_2} \rightarrow E \rightarrow \Oo_{l_1}(-2)\cup \Oo_{l_2}(-1) \rightarrow 0,
\end{equation}
and in particular the filtrations in the lemma (2.1) of \cite{IM}, coincide in $q$. Thus, $h^0(E|_{C_H})=2$.

Assume that $|Z\cap H|\leq 2$. If $C_H$ is smooth, we obtain in a similar way as above that $E_{C_H}$ is either $\Oo_{C_H}(-2p)\oplus \Oo_{C_H}$ or $\Oo_{C_H}(-p)\oplus \Oo_{C_H}(-p)$ and thus $h^0(E|_{C_H})\leq 1$. When $H$ is a tangent plane section at $q\in Q$, we can also similarly show that $h^0(E|_{C_H})\leq 1$, except when $Z\cap H=\{x,y\}$ and $y=q$, say $x\in l_1$. In this case, we have
\begin{align*}
E|_{l_1}&\simeq \Oo_{l_1}(1)\oplus \Oo_{l_1}(-2), \text{   and}\\
E|_{l_2}&\simeq \Oo_{l_2}\oplus \Oo_{l_2}(-1).
\end{align*}
Since $y=q$ is the intersection point of $l_1$ and $l_2$, the sub-bundles $\Oo_{l_1}(1)$ and $\Oo_{l_2}$ in (\ref{eq4}) do not coincide at $y$. So $h^0(E|_{C_H})=1$.
\end{proof}

Since we have $k \choose 3$ hyperplanes that meet $Z$ at 3 points and thus $S(E)$ has at least $k\choose 3$ singular points. Thus we have the following statement.
\begin{proposition}
For a Hulsbergen bundle $E\in \Mm(k)$, $S(E)$ is a hypersurface of degree $k-1$ in $\PP_3^*$ with $k \choose 3$ singular points.
\end{proposition}

\subsection{}
If $c_2=1$, then there is no stable vector bundles. In fact, it can be shown \cite{Huh} that there exists a unique strictly semi-stable vector bundle $E_0:=\Oo_Q(-1,0)\oplus \Oo_Q(0,-1)$. Since $h^0(E_0)=h^1(E_0(-1,-1))=0$, we have $h^0(E_0|_{C_H})=0$
for all $H\in \PP_3^*$. Hence, if we extend the concept of the jumping conic to semi-stable bundles, we can say that there is no jumping conic of $E_0$. It is consistent with the fact that $S(E_0)$ is a hypersurface of degree 0 in $\PP_3^*$.

\subsection{}
If $c_2=2$, then $S(E)$ is a hyperplane in $\PP_3^*$. So the map $S$ is from $\Mm(2)$ to $\PP_3$. It was shown in \cite{Huh} that $S$ extends to an isomorphism
$$\overline{S} : \overline{\Mm}(2) \rightarrow \PP_3,$$
where $\overline{\Mm}(2)$ is the compactification of $\Mm (2)$ in the sense of Gieseker \cite{Gieseker}, whose boundary consists of non-locally free sheaves with the same numeric invariants. In fact, for $E\in \overline{\Mm}(2)$, we have $h^0(E(1,1))=3$ and can define a morphism from $\PP_2 \simeq \PP H^0(E(1,1))$ to the Grassmannian $Gr(1,3)$, sending a section $s$ to the line in $\PP_3$ containing the two zeros of $s$. The image of this map can be shown to be a 2-cycles of $Gr(1,3)$ corresponding to the unique point in $\PP_3$. $\overline{S}$ maps $E$ to this uniquely determined point. Moreover, $\Mm (2)$ maps to $\PP_3 \backslash Q$ via $S$ and in particular, $S(E)$ determines $E$ completely. Let $Z$ be a 0-cycle on $Q$ with length 2 such that the support of $Z$ does not lie on a line in $Q$ and consider an extension family $\PP(Z)$ of $E$, admitting the following exact sequence,
$$0\rightarrow \Oo_Q \rightarrow E(1,1) \rightarrow I_Z(1,1) \rightarrow 0.$$
Then, $\PP(Z)\simeq \PP_1$ is the secant line of $Q$ passing through the support of $Z$. From this description, it can be easily checked that $H\in S(E)$ if and only if $E|_{C_H}\simeq \Oo_{C_H}\oplus \Oo_{C_H}(-2p)$, which is consistent with the fact that $S(E)$ is smooth.

\subsection{}If $c_2=3$, we have a map $S: \Mm(3) \rightarrow |\Oo_{\PP_3^*}(2)|\simeq \PP_9$, where $S(E)$ is a quadric in $\PP_3^*$. $E(1,1)$ is fitted into the following exact sequence,
\begin{equation}\label{ext}
0\rightarrow \Oo_Q \rightarrow E(1,1) \rightarrow I_Z(1,1) \rightarrow 0,
\end{equation}
with a 0-cycle $Z$ on $Q$ with length 3. If $Z$ is contained in a line on $Q$, then $E$ contains $\Oo_Q(0,-1)$ or $\Oo_Q(-1,0)$ as a sub-bundle, contradicting to the stability of $E$. Thus there exists a unique hyperplane $H$ in $\PP_3$ containing $Z$.
\begin{remark}
Conversely, if $Z$ is not contained in any line on $Q$, then it can be easily shown from the standard computation that any sheaf $E$ admitting an exact sequence (\ref{ext}) is semi-stable. In fact, if a subscheme of length 2 of $Z$ is contained in a line on $Q$, any sheaf $E$ admitting (\ref{ext}) is strictly semi-stable.
\end{remark}
Now let us consider a map
$$\eta_E : \PP_1 \simeq \PP H^0(E(1,1)) \rightarrow Gr(2,3) \simeq \PP_3^*,$$
sending a section $s\in H^0(E(1,1))$ to the projective plane in $\PP_3$ containing a 0-cycle $Z$ in the exact sequence (\ref{ext}), which is obtained from $s$. Before proving that $\eta_E$ is a constant map, we suggest a different proof of the fact that $S(E)$ is a quadric cone in $\PP_3^*$.
\begin{proposition}
For $E\in \Mm(3)$, $S(E)$ is a quadric cone in $\PP_3^*$ with a vertex point.
\end{proposition}
\begin{proof}
Let $s$ be a section of $E(1,1)$ from which $E(1,1)$ admits an exact sequence (\ref{ext}) for a 0-dimensional cycle $Z$ of length 3. Let $Z=\{z_1, z_2, z_3\}$. If $H_s$ be a hyperplane in $\PP_3$ containing $Z$, then $E|_{C_{H_s}}$ admits an exact sequence,
$$0\rightarrow \Oo_{C_{H_s}}(p) \rightarrow E|_{C_{H_s}} \rightarrow \Oo_{C_{H_s}}(-3p)\rightarrow 0,$$
where $p$ is a point on $C_{H_s}$. It splits since $H^1(\Oo_{C_{H_s}}(-6p))=0$. Thus $E_{C_{H_s}}$ is isomorphic to $\Oo_{C_{H_s}}(p)\oplus \Oo_{C_{H_s}}(-3p)$ and in particular, $H_s\in S(E)$. Similarly, if $H$ contains only 2 points of $Z$, then $H\in S(E)$. It can be shown that $H\not \in S(E)$ if $H$ contains only 1 point of $Z$. Let us consider a hyperplane $H(z_1)$ in $\PP_3^*$, whose points correspond to the hyperplanes in $\PP_3$ containing $z_1$. From the previous argument, we know that the intersection of $H(z_1)$ with $S(E)$ consists of 2 straight lines whose intersection point corresponds to the hyperplane $H_s$. If $S(E)$ is a smooth quadric, then $H(z_1)$ is the tangent plane of $S(E)$ at $H_s$. Similarly, we can define $H(z_i)$, $i=2,3$, and they would also become the tangent plane of $S(E)$ at $H_s$, which is absurd. We can similarly derive a contradiction in the case when $Q$ is a hyperplane in $\PP_3^*$ with multiplicity 2. Let us assume that $S(E)$ consists of two hyperplanes meeting at a line $l$. Clearly, $H_s$ lies in $l$. There are 3 lines on $S(E)$ corresponding to the hyperplanes containing 2 points of $Z$ and they are exactly the intersection of $H(z_i)$'s. Hence there is a hyperplane of $S(E)$ that contains two intersecting lines of $H(z_i)$'s. It is impossible since the two intersecting lines of $H(z_i)$ with $S(E)$ lie on different components of $S(E)$. Thus $Q$ is a quadric cone with a vertex point.
\end{proof}
\begin{corollary}
For $E\in \Mm (3)$, the map $\eta_E$ is a constant map to the vertex point of $S(E)$.
\end{corollary}
\begin{proof}
Using the notation in the proof of the preceding proposition, the planes $H(z_i)$ meet with $S(E)$ at two different lines. The only possibility is that $H_s$ is the vertex point of $S(E)$. Now we get the assertion since this argument is valid for all sections of $E(1,1)$.
\end{proof}
\begin{remark}
The hyperplane $H$ corresponding to the vertex point of $S(E)$ is the unique hyperplane for which $E|_{C_H}$ is isomorphic to $\Oo_{C_H}(-3p)\oplus \Oo_{C_H}(p)$, where $p$ is a point on $Q$. For the other hyperplanes in $S(E)$, $E|_{C_H}$ become $\Oo_{C_H}(-2p)\oplus \Oo_{C_H}$. \end{remark}
By sending $E\in \Mm(3)$ to the vertex point of $S(E)$, we can define a map
$$\Lambda^* : \Mm(3) \rightarrow \PP_3^*.$$
Let $p$ be a point in $\PP_3^*\backslash Q^*$, where $Q^*$ is the dual of $Q$, whose points correspond to the tangent planes of $Q$. We can pick a stable vector bundle $E$ fitted into the exact sequence (\ref{ext}) for a 0-cycle $Z$ of length 3 whose support lies in the hyperplane section corresponding to $p$. Then $E$ maps to the point $p$ via $\Lambda^*$. In the case when $p\in Q^*$, we can also choose a 0-cycle $Z$ for which there exists a stable vector bundle $E$ mapping to $p$. Thus $\Lambda^*$ is surjective and its generic fibres are 4-dimensional.

Now let us consider the determinant map
$$\lambda_E:\wedge^2H^0(E(1,1))\rightarrow H^0(\Oo_Q(1,1)).$$
Since $h^0(E(1,1))=2$, the dimension of the domain is 1-dimensional.
\begin{lemma}
$\lambda_E$ is injective.
\end{lemma}
\begin{proof}
We follow the argument in the proof of the lemma (6.6) in \cite{OPP}. Let $s_1$, $s_2$ be two linearly independent sections of $E(1,1)$. Assume that $s_1\wedge s_2$ maps to $0$ via $\lambda_E$. It would generate a line subbundle $L$ of $E(1,1)$ with $h^0(L)=2$. The only choices for $L$ is $\Oo_Q(0,1)$ and $\Oo_Q(1,0)$, and both contradict the stability of $E$.
\end{proof}
Let us define $q_E$ to be the point in $\PP_3^* \simeq \PP H^0(\Oo_Q(1,1))$ corresponding to the image of $\lambda_E$. Since $E(1,1)$ is fitted into the exact sequence (\ref{ext}), $H^0(E(1,1))$ can be considered as the direct sum of $H^0(\Oo_Q)$ and $H^0(I_Z(1,1))$, so $\wedge^2 H^0(E(1,1))$ is isomorphic to $H^0(I_Z(1,1))$. From the long exact sequence of cohomology of the exact sequence,
$$0\rightarrow I_Z(1,1) \rightarrow \Oo_Q(1,1) \rightarrow \Oo_Z \rightarrow 0 ,$$
$H^0(I_Z(1,1))$ is embedded into $H^0(\Oo_Q(1,1))$. This embedding is determined by the injection of $H^0(\Oo_Z)^*$ into $H^0(\Oo_Q(1,1))^*$, i.e. the hyperplane in $\PP_3$ containing $Z$. We know from the preceding corollary that this hyperplane is independent on the sections of $E(1,1)$. Thus the embedding of $H^0(I_Z(1,1))$ into $H^0(\Oo_Q(1,1))$ is independent on $Z$ and it would give the same map as $\lambda_E$. As a quick consequence of this argument, we obtain that the image of $\lambda_E$ corresponds to the unique hyperplane in $\PP_3$ containing $Z$. In other words, we obtain the following statement.
\begin{proposition}
$q_E$ is the vertex point of $S(E)$.
\end{proposition}
\begin{remark}
Let $f_Q$ be the polar map from $\PP_3$ to $\PP_3^*$ given by
\begin{equation}
[x_0, \cdots, x_3] \mapsto [\frac{\partial f}{\partial t_0}(x), \cdots, \frac{\partial f}{\partial t_3}(x)],
\end{equation}
where $f$ is the homogeneous polynomial of degree 2 defining $Q$. Then we have a surjective map from $\Mm(3)$ to $\PP_3$,
$$\Lambda := f_Q^{-1}\circ \Lambda^* : \Mm(3) \rightarrow \PP_3.$$
For $E\in \Mm(3)$, let $H_E$ be the hyperplane of $\PP_3$ corresponding to $q_E$. Note that $C_{H_E}=H_E\cap Q$ is the set of points $p\in Q$ for which $\Lambda(E)$ is contained in the tangent plane of $Q$ at $p$. Thus we can define the map $\Lambda$ by sending $E$ to the intersection point of the tangent planes at the support of $Z$ in the exact sequence (\ref{ext}), which is independent on the choice of a section of $E(1,1)$.
\end{remark}
Recall that the set of singular quadrics in $\PP_3^*$ is the discriminant hypersurface $\Dd _2$ in $\PP_9$ defined by the equation $\det (\Aa)=0$, where $\Aa$ is a symmetric $4\times 4$-matrix. By differentiating, we know that the singular points of $\Dd_2$ are defined by the determinants of $3\times 3$-minors of $\Aa$, i.e. the singular points of $\Dd_2$ correspond to the singular quadrics of rank $\leq 2$. Let $\Dd_2^0$ be the smooth part of $\Dd_2$. Then we have the following picture,
\begin{equation}
\xymatrix{ \Mm(3) \ar[r]^S\ar[rd]_{\Lambda^*} & \Dd_2^0\ar[d]\\
 & \PP_3^*,}
\end{equation}
where $\Dd_2^0$ is an open Zariski subset of a quartic hypersurface $\Dd_2$ of $\PP_9$ and the vertical map sends a singular quadric of rank 3 to its vertex point.

Let $E\in (\Lambda^*)^{-1}(q_E)$ with $q_E\not \in Q^*$. Thus $H_E$ is not a tangent plane of $Q$ and so $C_{H_E}$ is a smooth conic on $H_E$. Let $\PP_2^*$ be the image of $H_E$ via the polar map $f_Q$, which is a hyperplane of $\PP_3^*$, not containing $q_E$. Then $\PP_2^*$ contains the dual conic $C_{H_E}^*$ of $C_{H_E}$ via $f_Q|_{H_E}$. Let $\pi_{q_E}$ be the projection map from $\PP_3^*$ to $\PP_2^*$ at $q_E$. Then we can assign a smooth conic $C(E):=\pi_{q_E}(S(E)) \subset \PP_2^*$ to $E$, i.e. we have a map
$$\pi_{q_E} : (\Lambda^*)^{-1}(q_E) \rightarrow |\Oo_{\PP_2}^*(2)|\simeq \PP_5.$$
Clearly, $C(E)\not= C_{H_E}^*$.

Let us fix a general hyperplane $H$ of $\PP_3$. For a 0-cycle $Z$ with length 3 contained in $C_H\simeq \PP_1$, we can consider an extension space $\PP (Z):=\PP \Ext^1 (I_Z(1,1), \Oo_Q)\simeq \PP_2$. Note that the Hilbert scheme parametrizing 0-cycles on $C_H$ with length 3, $\PP_1^{[3]}$, is isomorphic to $\PP_3$. Let us define
$$\Uu := R^1{p_1}_* (\Ii \otimes {p_2}^*\Oo_Q(-1,-1)),$$
where $p_1, p_2$ are the projections from $\PP_3 \times Q$ to each factors and $\Ii$ is the universal ideal sheaf of $\PP_3 \times Q$. We can easily find that $\Uu$ is a vector bundle on $\PP_3$ of rank 3 and the fibre of $\PP (\Uu^*)$ at $Z\in \PP_3$ is $\PP (Z)$. Then we have a rational map from $\PP (\Uu^*)$ to $\Mm (3)_q:=(\Lambda^*)^{-1}(q)$, and eventually to $\PP_5$ after the composition with $\pi_q$, where $q$ corresponds to $H$. In particular, the dimension of the image of $\PP(\Uu^*)$ is less than 5 since the dimension of $\Mm(3)_q$ is 4.
\begin{equation}
\xymatrix{ \PP (\Uu^*)\ar[rd] \ar[rr] &&\PP_5 \\
&\Mm(3)_q\ar[ru]}
\end{equation}
For a general 0-cycle $Z=\{ z_1, z_2, z_3\}$ on $C_H$, let $p_{ij}\in \PP_2^*$ be the point corresponding to the line containing $z_i$ and $z_j$. The conic $C(E)$ contains $p_{ij}$ and so the image of $\PP (Z)$ is contained in the projective plane in $\PP_5$ parametrizing all the conics passing through three points $p_{ij}$. Let $Z^*=\{ z_1^*, z_2^*, z_3^*\}$ be the dual lines on $\PP_2^*$ of $Z$, then $p_{ij}$ is the intersection point of $z_i^*$ and $z_j^*$. If we choose linear forms $0\not= Z_i\in H^0(\Oo_{\PP_2^*}(1))$ which vanish on $z_i^*$, then from the previous statement, $\pi_q \circ S$ is defined by
\begin{align*}
\pi_q \circ S : \PP (Z) &\rightarrow |\Oo_{\PP_2^*}(2)|\\
(c_1,c_2,c_3) &\mapsto f_1Z_2Z_3+f_2Z_1Z_3+f_3Z_1Z_2,
\end{align*}
where $(c_1,c_2,c_3)$ is the coordinates from the identification of $\PP(Z)$ with $\PP H^0(\Oo_Z)^*$ and $f_i$'s are homogeneous polynomials of $c_j$'s.

\begin{proposition}
For a general 0-cycle $Z$, the map $\pi_q\circ S$ from $\PP (Z)$ to $\PP_5$ sending $E$ to $\pi_q(S(E))$, is a linear embedding.
\end{proposition}
\begin{proof}
From the previous argument, it is enough to check that $f_i$'s are linearly independent linear polynomials. In fact we can prove that $f_i\equiv c_i$ for all $i$.

Recall that $\mathbf{I}$ is the incidence variety in $Q \times \PP_3^*$ with the projections $\pi_1$ and $\pi_2$. Then we have an isomorphism,
$$h: \Oo_{\PP_3^*} \rightarrow {\pi_2}_*{\pi_1^*}I_Z((0,0), 3),$$
given by the multiplication with $Z_1Z_2Z_3$. Here, $\Oo_{\mathbf{I}}((a,b),c)$ is the sheaf $\pi_1^*\Oo_Q(a,b) \otimes \pi_2^*\Oo_{\PP_3^*}(c)$ on $\mathbf{I}$. Note that ${\pi_2}_*{\pi_1^*}I_Z$ is the ideal sheaf of functions on $\PP_3^*$, vanishing on the lines $z_i^*$. From the canonical homomorphisms,
\begin{align*}
\Ext^1 (I_Z(1,1), \Oo_Q) &\rightarrow \Ext^1 (\pi_1^*I_Z(1,1), \Oo_{\mathbf{I}})\\
&\rightarrow \Ext^1 (\pi_1^*I_Z((0,0), 3), \Oo_{\mathbf{I}}((-1,-1),3)),
\end{align*}
we can assign to an element $\epsilon \in \Ext^1 (I_Z(1,1), \Oo_Q)$, an extension
\begin{equation}\label{ep}
0\rightarrow \Oo_{\mathbf{I}}((-1,-1),3) \rightarrow \pi_1^*E((0,0),3) \rightarrow \pi_1^*I_Z((0,0), 3) \rightarrow 0.
\end{equation}
From the long exact sequence of cohomology of (\ref{ep}), we obtain
$$H^0(\Oo_{\PP_3}^*) \rightarrow H^0(\pi_1^*I_Z((0,0), 3)) \rightarrow H^1(\Oo_{\mathbf{I}}((-1,-1),3)) \simeq H^0(\Oo_{\PP_3^*}(2)),$$
and let $\pi(\epsilon)$ be the image of $1\in H^0(\Oo_{\PP_3^*})$. Then we can define a homomorphism
\begin{equation}
\pi : \Ext^1 (I_Z(1,1), \Oo_Q) \rightarrow H^0(\Oo_{\PP_3^*}(2)),
\end{equation}
by sending $\epsilon$ to $\pi(\epsilon)$.

From the inclusion $I_{z_i} \hookrightarrow I_Z$, we have a natural injection from
$$\Ext^1 (I_{z_i}(1,1), \Oo_Q)\simeq \CC \hookrightarrow \Ext^1(I_Z(1,1), \Oo_Q)$$
whose image is $H^0(\Oo_{z_i})^*$. It can be easily checked that any element in the image is mapped to $H^0(\Oo_{\PP_3^*}(2))$ by the multiplication with $(Z_1Z_2Z_3)/Z_i$. Thus $\pi$ is defined by sending $(c_1,c_2,c_3)$ to $c_1Z_2Z_3+c_2Z_1Z_3+c_3Z_1Z_2$.

When we take the direct image of (\ref{ep}), then we obtain
\begin{align*}
{\pi_2}_*\pi_1^*E((0,0),3) &\rightarrow {\pi_2}_*\pi_1^*I_Z((0,0), 3) \rightarrow R^1 {\pi_2}_* \Oo_{\mathbf{I}}((-1,-1),3)\\
&\rightarrow  R^1{\pi_2}_*\pi_1^*E((0,0),3)  \rightarrow R^1{\pi_2}_*\pi_1^*I_Z((0,0), 3) \rightarrow 0.
\end{align*}
Note that ${\pi_2}_*\pi_1^*I_Z((0,0), 3) \simeq \Oo_{\PP_3^*}$, $R^1 {\pi_2}_* \Oo_{\mathbf{I}}((-1,-1),3) \simeq \Oo_{\PP_3^*}(2)$ and the second map in the sequence above, is given by the multiplication with $\pi( \epsilon)$. As an analogue of the result in \cite{Hulek}, we can easily check that $R^1{\pi_2}_*\pi_1^*E((0,0),3)$ is isomorphic to $\theta_E(4)$ and its support is $S(E)$. On the other hand, the support of $R^1{\pi_2}_*\pi_1^*I_Z((0,0), 3)$ is contained in $\{p_{ij}\}$ and thus the support of $S(E)$ is same as the support of $\{\pi(\epsilon)=0\}$. Because of the same degree, they are the same.
\end{proof}
\begin{remark}
Using the argument as in the similar statement on the projective plane in \cite{Hulsbergen}, we can prove that a sheaf $E\in \PP (Z)$ with the coordinates $(c_1,c_2,c_3)$ is locally free if and only if $c_i\not= 0$ for all $i$. Thus from the proof of the preceding proposition, we can observe that the conic corresponding to the image of $E$ is smooth if and only if $E$ is locally free. Note that the secant variety $V_3$ of the Veronese surface in $\PP_5$ is a cubic hypersurface. The intersection of the image of $\PP (Z)$ with $V_3$ are the 3 lines, which are the image of non-locally free sheaves in $\PP (Z)$.
\end{remark}

We can see that the same statement holds for arbitrary hyperplane section $H\in \PP_3^*$. If $H\in Q^*$, $Q^*$ the dual conic of $Q$, then $C_H=l_1\cup l_2$. Because of the stability condition, our 0-cycles of length 3 associated to $E$ with $\Lambda^*(E)\in Q^*$ cannot have its support only on $l_1$ nor $l_2$. So the family of 0-cycles we consider, is isomorphic to the two copies of $\PP_1^{[2]}\times \PP_1$. Let us denote
$$\Mm(3) = \Mm^0(3)\coprod \Mm^{1}(3) \coprod \Mm^{2}(3),$$
where $\Mm^0(3)=(\Lambda^*)^{-1}(\PP_3^* \backslash Q^*)$ and $\Mm^{i}(3)$'s are the two irreducible components of $(\Lambda^*)^{-1}(Q^*)$ whose 0-cycles have two points of its support on the ruling equivalent to $l_i$.

First let us assume that $H\not \in Q^*$. Let $V\subset \PP_5$ be the image of $\PP (\Uu^*)$ and $v\in V$ be a general point in $V$. Then there exists three points $z_i$'s on $C_H$ and $c_i$'s for which we have $v=c_1Z_1+c_2Z_2+c_3Z_3$. Since $z_i\in C_H$, the lines $Z_i$'s are tangent to the dual conic $C_H^*$, i.e. $Z_i$'s is a circumscribed triangle around $C_H^*$. Note that $Z_i$'s is a inscribed triangle in $v$. Thus $V$ is the closure of the family of conics Poncelet related to $C_H^*$ (see section 2 in \cite{D}). From the classical result, $V$ is a hypersurface in $\PP_5$ and the generic fibre of the map $\PP (\Uu^*) \rightarrow V$ is isomorphic to $\PP_1$. In fact, from the remark (2.2.3) in \cite{D}, $V$ is isomorphic to a hypersurface of degree 4, $H_4$ in the space of conics, given by the condition $c_2^2-c_1c_3=0$, where 
$$\det (A-tI_3)=(-t)^3+c_1(-t)^2+c_2(-t)+c_3,$$
is the characteristic polynomial of a symmetric matrix $A$ defining a conic. 

Let $E\in \Mm^1(3)$ and so $H\in Q^*$. If we define $V$ as before and let $v\in V$ be a general conic, then $v$ pass through the dual point $p_1\in \PP_2$ of $l_1$. Let us fix a conic $v$ passing through $p_1$. If we choose $q_1 \in v$ not equal to $p_1$, then consider a line $l$ passing through $q_1$ and the dual point $p_2$ of $l_2$. Let $q_2$ be the other intersection point of $l$ with $v$. Then the dual points corresponding to the lines $\overline{p_1q_1}, \overline{q_1q_2}, \overline{q_2p_1}$ is a 0-cycle $Z$ mapping to $v$. It depends on the choice of $q_1$. Thus, $V$ is isomorphic to a hyperplane in $\PP_5$ and the generic fibre of the map from $\PP (\Uu^*)$ is again isomorphic to $\PP_1$. We have the same argument for $\Mm^2 (3)$.

As a direct consequence, $\Mm(3)_q$ is isomorphic to an open Zariski subset of a hyperplane in $\PP_5$ and thus we have the following proposition.
\begin{proposition}
$\Mm(3)$ admits a fibration over $\PP_3^*$ whose fibre over $H\in \PP_3^*$ is isomorphic to  
\begin{enumerate}
\item an open Zariski subset $H_4 \cap (\PP_5\backslash V_3)$ of a $H_4$, where $V_3$ is the secant variety of the Veronese surface $S\subset \PP_5$ and $H_4$ is a hypersurface of degree 4 consisting of conics Poncelet related to $Q\cap H$, if $H\in \PP_3^*\backslash Q^*$;
\item the union of two varieties $H_i \cap (\PP_5 \backslash V_3)$, $i=1,2$, where $H_i$ is the hyperplane in the space of conics which pass through a point $p_i$ dual to the line $l_i\subset H$, where $Q\cap H=l_1+l_2$, if $H\in Q^*$. 
\end{enumerate}
\end{proposition}

\begin{remark}
In fact, we can obtain differently the old result of Darboux on the Poncelet related conics in the case of triangles. We know that we have $\dim V \leq \dim \Mm(3)_q=4$. Assume that $C_H$ is a smooth conic. Let $\triangle_2$ be the subscheme of $C_H^{[3]}$ whose points are 0-cycles with at most 2 points as their supports. Similarly, we can define $\triangle_3 \subset \triangle_2$. Let $Z\in \triangle_2$, say $Z=\{x,x,y\}$. The map $\PP (Z) \rightarrow \PP_5$ is naturally defined by sending $(c_1, c_2, c_3)$ to $(c_1+c_2)XY+c_3X^2$. From this observation, the image of $\PP_2$-bundle over $\delta_3$ is $C_H\subset \PP_5$ mapped by $|\Oo_{C_H}(2)|$. For $Z=\{x,x,y\}$, $\PP(Z)$ is mapped to the line passing through $X^2$ and $XY$. When $Y$ is moving along $C_H$, this line covers a projective plane $\PP_2(x)$ passing through the point $X^2\in C_H\subset \PP_5$. Let $D$ be the union of such projective planes over $x$ moving along $C_H$. In particular, $D$ is a subvariety of $V$ with dimension 3 and all the non-locally free sheaves in $\PP (\Uu^*)$ map to $D$. Also we have
$$V_3\cap V =D,$$
where $V_3$ is the secant variety of the Veronese surface in $\PP_5$. It also implies that $V$ is a subvariety of $\PP_5$ with dimension 4.

Let us consider a fibre of the map $\PP (\Uu^*) \rightarrow V$ over $XY$ with $x,y\in C_H$. The image of the closure of this fibre via the projection to $C_H^{[3]}$ is isomorphic to $\PP_1$, parametrizing 0-cycles whose supports contain $x$ and $y$. In fact, there exists a unique component of the closure of the fibre, mapping to $\PP_1\subset C_H^{[3]}$. It implies that the closure of the fibre over a generic conic $v$ in $V$ is isomorphic to $\PP_1$ since there exists at most 1 point in $\PP (Z)$ that maps to $v$.

\end{remark}

In particular, the map $S: \Mm(k) \rightarrow \PP_N$ is an isomorphism to its image, when $k\leq 3$.
\begin{theorem}
The set of jumping conics of $E\in \Mm(k)$, determines $E$ uniquely when $k\leq 3$.
\end{theorem}

\providecommand{\bysame}{\leavevmode\hbox to3em{\hrulefill}\thinspace}
\providecommand{\MR}{\relax\ifhmode\unskip\space\fi MR }
\providecommand{\MRhref}[2]{%
  \href{http://www.ams.org/mathscinet-getitem?mr=#1}{#2}
}
\providecommand{\href}[2]{#2}


\begin{thebibliography}{10}

\bibitem{Barth2}
W.~Barth, \emph{Moduli of vector bundles on the projective plane}, Invent.
  Math. \textbf{42} (1977), 63--91. \MR{0460330 (57 \#324)}

\bibitem{Barth}
\bysame, \emph{Some properties of stable rank-{$2$} vector bundles on {${\bf
  P}\sb{n}$}}, Math. Ann. \textbf{226} (1977), no.~2, 125--150. \MR{0429896
  (55 \#2905)}

\bibitem{Buch}
N.~P. Buchdahl, \emph{Stable {$2$}-bundles on {H}irzebruch surfaces}, Math. Z.
  \textbf{194} (1987), no.~1, 143--152. \MR{871226 (88c:14022)}

\bibitem{C}
Fabrizio Catanese, \emph{Footnotes to a theorem of {I}. {R}eider}, Algebraic
  geometry ({L}'{A}quila, 1988), Lecture Notes in Math., vol. 1417, Springer,
  Berlin, 1990, pp.~67--74. \MR{1040551 (91d:14022)}

\bibitem{D}
Igor Dolgachev, \emph{Topics in classical algebraic geometry - part 1}, Lecture
  Notes, http://www.math.lsa.umich.edu/~idolga/topics1.pdf.

\bibitem{Gieseker}
D.~Gieseker, \emph{On the moduli of vector bundles on an algebraic surface},
  Ann. of Math. (2) \textbf{106} (1977), no.~1, 45--60. \MR{466475
  (81h:14014)}

\bibitem{Harris}
Joe Harris, \emph{Algebraic geometry}, Graduate Texts in Mathematics, vol. 133,
  Springer-Verlag, New York, 1995, A first course, Corrected reprint of the
  1992 original. \MR{1416564 (97e:14001)}

\bibitem{Huh}
S.~Huh, \emph{A moduli space of stable sheaves on a smooth quadric in
  $\mathbb{P}_3$}, Preprint, arXiv:0810.4392 [math.AG], 2008.

\bibitem{Hulek}
Klaus Hulek, \emph{Stable rank-{$2$} vector bundles on {${\bf P}\sb{2}$} with
  {$c\sb{1}$}\ odd}, Math. Ann. \textbf{242} (1979), no.~3, 241--266.
  \MR{545217 (80m:14011)}

\bibitem{Hulsbergen}
W.~Hulsbergen, \emph{Vector bundles on the projective plane}, Ph.D. thesis,
  Leiden, 1976.

\bibitem{IM}
M.~Id{\`a} and M.~Manaresi, \emph{Rank two vector bundles on {${\bf P}\sp n$}
  uniform with respect to some rational curves}, Arch. Math. (Basel)
  \textbf{51} (1988), no.~3, 266--273. \MR{960405 (89i:14011)}

\bibitem{Maruyama}
Masaki Maruyama, \emph{The theorem of {G}rauert-{M}\"ulich-{S}pindler}, Math.
  Ann. \textbf{255} (1981), no.~3, 317--333. \MR{615853 (82k:14012)}

\bibitem{OSS}
Christian Okonek, Michael Schneider, and Heinz Spindler, \emph{Vector bundles
  on complex projective spaces}, Progress in Mathematics, vol.~3, Birkh\"auser
  Boston, Mass., 1980. \MR{561910 (81b:14001)}

\bibitem{OPP}
W.~M. Oxbury, C.~Pauly, and E.~Previato, \emph{Subvarieties of {$\rm{SU}\sb
  C(2)$} and {$2\theta$}-divisors in the {J}acobian}, Trans. Amer. Math. Soc.
  \textbf{350} (1998), no.~9, 3587--3614. \MR{1467474 (98m:14034)}

\bibitem{Vitter}
Al~Vitter, \emph{Restricting semistable bundles on the projective plane to
  conics}, Manuscripta Math. \textbf{114} (2004), no.~3, 361--383.
  \MR{2076453 (2005e:14066)}

\end{thebibliography}
\end{document}